\documentclass[reqno]{amsart}
\usepackage{amsmath, amssymb, amsfonts}
\newtheorem{theorem}{Theorem}

\newtheorem{corollary}{Corollary}
\newtheorem{lemma}{Lemma}

\theoremstyle{remark}
\newtheorem{remark}{Remark}
\theoremstyle{definition}
\newtheorem{define}{Definition}

\newcommand{\CC}{\mathbb C}

\begin{document}

\title[On the real-analytic infinitesimal CR automorphism]{On the real-analytic infinitesimal CR automorphism of hypersurfaces of infinite type}

\author{Ninh Van Thu, Chu Van Tiep and Mai Anh Duc}

\thanks{ The research of the authors was supported in part by an NAFOSTED grant
of Vietnam. The research of the first author was supported in part by an NRF grant 2011-0030044 (SRC-GAIA) of the Ministry of Education, The Republic of Korea.}

\address{Center for Geometry and its Applications,
 Pohang University of Science and Technology,  Pohang 790-784, The Republic of Korea}
\email{thunv@postech.ac.kr, thunv@vnu.edu.vn}

\address{School of Physical and Mathematical Sciences, Nanyang Technological University, SPMS-04-01, 21 Nanyang Link,
Singapore 637371}
\email{vantiep001@e.ntu.edu.sg}
\address{Department of Mathematics, Hanoi
National University of Education, 136 Xuan Thuy str., Hanoi, Vietnam}
\email{ducphuongma@gmail.com}
\begin{abstract}
We consider a real smooth hypersurface $M\subset \mathbb C^2$, which is of D'Angelo infinite type at $p\in M$. The purpose of this paper is to show that the real vector space of tangential holomorphic vector field germs at $p$ vanishing at $p$ is either trivial or of real dimension $1$. 
\end{abstract}
\maketitle

\section{Introduction}

Let $(M,p)$ be a real $\mathcal{C}^1$-smooth hypersurface germ at $p\in \mathbb C^n$. A smooth vector field germ $(X,p)$ on $M$ is called \emph{a real-analytic infinitesimal CR automorphism germ at $p$ of $M$} if there exists a holomorphic vector field germ $(H,p)$ in $\mathbb C^n$ such that $H$ is tangent to $M$, i.e. $\mathrm{Re}~H$ is tangent to $M$, and $X=\mathrm{Re}~H\mid_M$. We denote by  $\mathrm{hol}_0(M,p)$ the real vector space of holomorphic vector field germs $(H,p)$ vanishing at $p$ which are tangent to $M$.

For a real hypersurface in $\mathbb C^n$, the real-analytic infinitesimal CR automorphism is not easy to describe explicitly; besides, it is unknown in most cases. For instance, the study of $\mathrm{hol}_0(M,p)$ of various hypersurfaces is given in \cite{Baouendi1, Chern, Kolar1, Stanton1, Stanton2}. However, these results are known for Levi nondegenerate hypersurfaces or more generally for Levi degenerate hypersurfaces of finite type. For various real $\mathcal{C}^\infty$-smooth hypersurfaces of D'Angelo infinite type in $\mathbb C^2$, explicit descriptions of $\mathrm{hol}_0(M,p)$ are given in \cite{By1, Kim-Ninh, Ninh}. 

In this paper we shall prove that $\mathrm{hol}_0(M,p)$ of a certain hypersurface of D'Angelo infinite type in $\mathbb C^2$ is either trivial or of real dimension $1$. To state the result explicitly, we need some notations and a definition. Taking the risk of confusion we employ the notations
$$
P'(z)=P_z(z)=\frac{\partial P}{\partial z}(z),~ f_z(z,t)=\frac{\partial f}{\partial z}(z,t),~f_t(z,t)=\frac{\partial f}{\partial t}(z,t)
$$
throughout the article. Also denote by $\Delta_r=\{z\in \mathbb C\colon |z|<r\}$ for $r>0$ and by $\Delta=\Delta_1$. A function $f$ defined on $\Delta_r~(r>0)$ is called to be \emph{flat} at the origin if $f(z)=o(|z|^n)$ for each $n\in \mathbb N$ (cf. Definition \ref{def1}).

The aim of this paper is to prove the following theorem. 
\begin{theorem} \label{Theorem3} Let $(M,0)$ be a real $\mathcal{C}^1$-smooth hypersurface germ at $0$ defined by the equation
$\rho(z) := \rho(z_1,z_2)=\mathrm{Re}~z_1+P(z_2)+ \mathrm{Im}~z_1 Q(z_2,\mathrm{Im}~z_1)=0$
satisfying the conditions: 
\begin{itemize}
\item[(1)] $P,Q$ are $\mathcal{C}^1$-smooth with $P(0)=Q(0,0)=0$,
\item[(2)] $P(z_2)>0$ for any $z_2 \not= 0$, and 
\item[(3)] $P(z_2),P'(z_2)$ are flat at $z_2=0$.
\end{itemize}
Then $\dim_{\mathbb R} \mathrm{hol}_0(M,p)\leq 1$.
\end{theorem}
\begin{remark} When $P,Q$ are $\mathcal{C}^\infty$-smooth, the condition $(3)$ simply says that $P$ vanishes to infinite order at $0$ and moreover $0$ is a point of D'Angelo infinite type.

\end{remark}
In the case $M$ is \emph{radially symmetric in $z_2$}, i.e. $P(z_2)=P(|z_2|)$ and $Q(z_2,t)=Q(|z_2|,t)$ for any $z_2$ and $t$, it is well-known that $iz_2\frac{\partial}{\partial z_2}$ is tangent to $M$ (see cf. \cite{By1}). Therefore, by Theorem \ref{Theorem3} one gets the following corollary, which is a slight generalization of the main result in \cite{By1}. 
\begin{corollary} Let $(M,0)$ be a real $\mathcal{C}^1$-smooth hypersurface germ at $0$ defined by the equation
$\rho(z) := \rho(z_1,z_2)=\mathrm{Re}~z_1+P(z_2)+ \mathrm{Im}~z_1 Q(z_2,\mathrm{Im}~z_1)=0$
satisfying the conditions: 
\begin{itemize}
\item[(1)] $P,Q$ are $\mathcal{C}^1$-smooth with $P(0)=Q(0,0)=0$,
\item[(2)] $P(z_2)=P(|z_2|),~Q(z_2,t)=Q(|z_2|,t)$ for any $z_2$ and $t$,
\item[(3)] $P(z_2)>0$ for any $z_2 \not= 0$, and 
\item[(4)] $P(z_2),P'(z_2)$ are flat at $z_2=0$.
\end{itemize} 
Then $\mathrm{hol}_0(M,0)=\{i\beta z_2\frac{\partial}{\partial z_2}\colon \beta\in \mathbb R\}$.
\end{corollary}
Next, we shall give an explicit description for real-analytic infinitesimal CR automorphisms of another class of real hypersurfaces in $\mathbb C^2$.

Let $a(z)=\sum_{n=1}^\infty a_n z^n$ be a nonzero holomorphic function defined on $\Delta_{\epsilon_0}~(\epsilon_0>0)$ and let $p,q$ be $\mathcal{C}^1$-smooth functions defined respectively on $(0,\epsilon_0)$ and $[0,\epsilon_0)$ satisfying that $q(0)=0$ and that $g(z),g'(z)$ are flat at $0$, where $g$ is a $\mathcal{C}^1$-smooth function given by
\[  
g(z)=
\begin{cases} e^{p(|z|)} &~\text{if}~0<|z|<\epsilon_0\\
0 &~\text{if}~z=0.
\end{cases}
\]

Denote by $M(a,\alpha, p,q)$ the germ at $(0,0)$ of a real hypersurface defined by
$$
\rho(z_1,z_2):= \mathrm{Re}~z_1+P(z_2)+f(z_2,\mathrm{Im}~z_1) =0,
$$
where $f$ and $P$ are respectively defined on $\Delta_{\epsilon_0}\times (-\delta_0,\delta_0)$ ($\delta_0>0$ small enough) and $\Delta_{\epsilon_0}$ by
\[
 f(z_2,t)=\begin{cases}
 -\frac{1}{\alpha}\log \Big|\frac{\cos \big(R(z_2)+\alpha t\big)}{\cos (R(z_2))} \Big| &~\text{if}~ \alpha\ne 0\\
 \tan(R(z_2))t  &~\text{if}~ \alpha =0,
\end{cases} 
 \]
 where $R(z_2)=q(|z_2|)- \mathrm{Re}\big(\sum_{n=1}^\infty\frac{a_n}{n} z_2^n\big)$ for all $z_2\in \Delta_{\epsilon_0}$,
and
\begin{equation*}
\begin{split}
  P(z_2)=
 \begin{cases}
\frac{1}{\alpha} \log \Big[ 1+\alpha P_1(z_2)\Big]~&\text{if}~ \alpha \ne 0\\
 P_1(z_2) ~&\text{if}~ \alpha=0,
\end{cases}
\end{split}
\end{equation*}
where 
\begin{equation*}
\begin{split}
P_1(z_2)=\exp\Big[p(|z_2|)+\mathrm{Re}\Big(\sum_{n=1}^\infty  \frac{a_n}{in}z_2^n\Big ) -\log \big|\cos\big(R(z_2)\big)\big|   \Big]
\end{split}
\end{equation*}
for all $z_2\in \Delta_{\epsilon_0}^*$ and $P_1(0)=0$.

 It is easily checked that $M(a,\alpha,p,q)$ is $\mathcal{C}^1$-smooth and moreover $P(z_2),P'(z_2)$ are flat at $0$. Furthermore, we note that $q,p$ can be chosen, e.g. $q(t)=0$ and $p(t)=-\frac{1}{t^\alpha}~(\alpha>0)$ for all $t>0$ , so that $P,R$ are $\mathcal{C}^\infty$-smooth in $\Delta_{\epsilon_0}$ and $P$ is flat at $0$, and hence $M(a,\alpha, p,q)$ is $\mathcal{C}^\infty$-smooth and of D'Angelo infinite type.  

It follows from Theorem \ref{T2} in Appendix that the holomorphic vector field
$$
H^{a,\alpha} (z_1,z_2):=L^\alpha (z_1) a(z_2)\frac{\partial }{\partial z_1}+iz_2\frac{\partial }{\partial z_2},
$$
where 
\[
L^\alpha(z_1)=
\begin{cases}
\frac{1}{\alpha}\big(\exp(\alpha z_1)-1\big)&\text{if}~ \alpha \ne 0\\
z_1 &\text{if}~ \alpha =0,
\end{cases}
\]
is tangent to $M(a,\alpha,p,q)$. Hence, by Theorem \ref{Theorem3} we obtain following corollary.  
\begin{corollary}
 $\mathrm{hol}_0\big(M(a,\alpha, p,q),0\big)=\{\beta H^{a,\alpha}\colon \beta \in \mathbb R\}$.
\end{corollary}

This paper is organized as follows. In Section $2$, we recall several definitions and give several technical lemmas. Next, the proof of Theorem \ref{Theorem3} is given in Section $3$. Finally, a theorem is pointed out in Appendix.

%\begin{Acknowlegement} The author would like to thank Prof. Do Duc Thai for his precious discussions on this material.
%\end{Acknowlegement}

\section{Preliminaries}
In this section, we shall recall several definitions and introduce two technical lemmas used in the proof of Theorem \ref{Theorem3}. In what follows, $\lesssim$ and $\gtrsim$ denote inequalities up to a positive constant. In addition, we use $\approx $ for the combination of $\lesssim$ and $\gtrsim$.
\begin{define}\label{def1}
A function $f:\Delta_{\epsilon_0}\to \mathbb C$ ($\epsilon_0>0$) is called to be \emph{flat} at $z=0$ if for each $n\in\mathbb N$ there exist positive constants $C,\epsilon >0$, depending only on $n$, with $0<\epsilon<\epsilon_0$ such that
$$
|f(z)|\leq C|z|^n
$$
for all $z\in \Delta_{\epsilon}$.
\end{define}

We note that in the above definition we do not need the smoothness of the function $f$. For example, the following function
\[
f(z)=
\begin{cases}
\frac{1}{n} e^{-\frac{1}{|z|^2}} &~\text{if}~\frac{1}{n+1}<|z|\leq \frac{1}{n}~,n=1,2,\ldots\\
0 &~\text{if}~z=0
\end{cases}
\]
is flat at $z=0$ but not continuous on $\Delta$. However, if $f\in \mathcal{C}^\infty(\Delta_{\epsilon_0})$ then it follows from the Taylor's theorem that $f$ is flat at $z=0$ if and only if
$$
\frac{\partial^{m+n}}{\partial z^m\partial \bar z^n}f(0)=0
$$ 
for every $m,n\in \mathbb N$, i.e., $f$ vanishes to infinite order at $0$. Consequently, if $f\in \mathcal{C}^\infty(\Delta_{\epsilon_0})$ is flat at $0$ then $\frac{\partial^{m+n} f}{\partial z^m\partial \bar z^n}$ is also flat at $0$ for each $m,n\in \mathbb N$.

Let $F$ be a $\mathcal{C}^1$-smooth complex-valued function defined in a neighborhood $U$ of the origin in the complex plane. We consider the autonomous dynamical system 
\begin{equation}\label{eq??1}
\frac{dz}{dt}=F(z),~z(0)=z_0\in U.
\end{equation}
First of all, let us recall several definitions.
\begin{define} A state $\hat z\in U$ is called an equilibrium of (\ref{eq??1}) if $F(\hat z)=0$.
\end{define}

\begin{define} An equilibrium, $\hat z$, of (\ref{eq??1}) is called locally asymptotically stable if for all $\epsilon>0$ there exists $\delta>0$ such that $|z_0-\hat z|<\delta$ implies that $|z(t)-\hat z|<\epsilon$ for all $t\geq 0$ and $\lim_{t\to +\infty} z(t)=0$. 
\end{define}

The following lemma is a generalization of Lemma $3$ in \cite{Kim-Ninh} and plays a key role in the proof of Theorem \ref{Theorem3}.
\begin{lemma}\label{Al3} Let $P:\Delta_{\epsilon_0}\to \mathbb R$ be a $\mathcal{C}^1$-smooth function satisfying that $P(z)>0$ for any $z\in \Delta^*_{\epsilon_0}$ and that $P$ is flat at $0$. If $a, b$ are complex numbers and if $g_0, g_1, g_2$ are $\mathcal{C}^1$-smooth functions defined on $\Delta_{\epsilon_0}$ satisfying:
\begin{itemize}
\item[(A1)] $g_0(z) = O(|z|)$, $g_1(z) = O(|z|^\ell)$, and $g_2(z) = o(|z|^m)$, and
\item[(A2)] $\text{Re} \Big[a z^m+\frac{1}{P^n(z)}\Big(b z^\ell\big(1+g_0(z)\big) \frac{P'(z)}{P(z)}
+g_1(z) \Big)\Big]= g_2(z)$ for every $z \in \Delta^*_{\epsilon_0}$
\end{itemize}
for any nonnegative integers $\ell, m$ and $n$ except for the following two cases
\begin{itemize}
\item[(E1)] $\ell=1$ and $\text{Re }b = 0$, and
\item[(E2)] $m=0$ and $\text{Re } a = 0$
\end{itemize}
then $ab=0$.
\end{lemma}

\begin{proof} We shall prove the lemma by contradiction.  Suppose that there exist
non-zero complex numbers $a,b\in \mathbb C^*$ such that the identity in (A2) holds with the smooth functions $g_0, g_1$, and $g_2$ satisfying the growth conditions specified in (A1).

Denote by $F(z):=\dfrac{1}{2}\log P(z) $ for all $z\in \Delta^*_{\epsilon_0}$ and by $f(z):=bz^\ell(1+g_0(z))$ for all $z\in \Delta_{\epsilon_0}$.
\medskip

\noindent
{\bf Case 1.} {\boldmath $\ell=0$:}

Let $\gamma: [0,\delta_0)\to \Delta_{\epsilon_0}~(\delta_0>0)$ be the solution of the initial-value problem 
$$ 
\frac{d\gamma(t)}{dt}=b+bg_0(\gamma(t)),\quad \gamma(0)=0.
$$
Let us denote by $u(t):=F(\gamma(t)),~0<t< \delta_0$. By (A2), it follows that $u'(t)$ is bounded on the interval $(0, \delta_0)$. Integration shows that $u(t)$ is also bounded on $(0, \delta_0)$. But this is impossible since
$u(t)\to -\infty$ as $t\downarrow 0$.
\medskip

\noindent
{\bf Case 2.} {\boldmath $\ell=1$:}

By (E1), we have $b_1:=\text{Re} (b)\ne 0$. Assume momentarily that $b_1<0$. Let $\gamma: [t_0,+\infty)\to \Delta^*_{\epsilon_0}~(t_0>0)$ be the solution of the initial-value problem 
\begin{equation}\label{teq1}
\frac{d\gamma(t)}{dt}=b\gamma(t)\Big(1+g_0(\gamma(t)\big)\Big),~\gamma(t_0)=z_0\in \Delta^*_{\epsilon_0}.
\end{equation}

Thanks to \cite[Theorem 5]{Co}, the system (\ref{teq1}) is locally trajectory equivalent at the origin to the system
$$
\frac{dz}{dt}=b z(t).
$$
It is well-known that the origin is a locally asymptotically stable equilibrium of the above diffenrential equation. Therefore, 
we have $\gamma(t)\to 0$ as $t\to +\infty$. Moreover, we can assume that $|\gamma(t)|<r_1$ for every $t_0<t<+\infty$, where $r_1:=1/2$ if $\mathrm{Im}(b)=0$ and $r_1:=\min\{1/2,|b_1|/(4|\mathrm{Im}(b)|)\}$ if otherwise. This implies that $\mathrm{Re}\Big(b(1+g_0(\gamma(t)))\Big)<b_1/4<0$. Integration and a simple estimation tell us that 
$$
 |\gamma(t)|\leq |\gamma(t_0)|\exp\Big(b_1t/4\Big),~\forall t>t_0.
$$ 
Consequently, this in turn yields that $t\lesssim \log \frac{1}{|\gamma(t)|}$.

Denote by $u(t):= F(\gamma(t))$ for $t\geq t_0$. Then, it follows
from (A2) that $u'(t)$ is bounded on $(t_0,+\infty)$, and thus $|u(t)|\lesssim t$.
Therefore, there exists a constant $A>0$ such that
$ |u(t)|\leq A\log \dfrac{1}{|\gamma(t)|}$ for all $t>t_0$.
Hence we obtain, for all $t>t_0$, that
$\log P(\gamma (t))=2u(t)\geq -2A \log \dfrac{1}{|\gamma(t)|} $, and thus
$$
P(\gamma(t))\geq |\gamma(t)|^{2A}, \; t\geq t_0.
$$
Hence we arrive at
$$
\lim_{t\to +\infty} \frac{P(\gamma(t))}{|\gamma(t)|^{2A+1}}=+\infty,
$$
which is impossible since $P$ is flat at $0$.  The case $b_1 >0$ is similar, with considering the side $t<0$ instead.
\medskip

\noindent
{\bf Case 3.} {\boldmath $\ell=k+1\geq 2$:}

Let $\gamma: [t_0, +\infty)\to \Delta^*_{\epsilon_0}~(t_0>0)$ be a solution of the initial-value problem 
\begin{equation}\label{teq2}
\frac{d\gamma(t)}{dt}=f(\gamma(t))=b\gamma^{k+1}(t)\Big(1+g_0(\gamma(t)\big)\Big),~\gamma(t_0)=z_0\in\Delta^*_{\epsilon_0}.
\end{equation}
According to \cite[Theorem 5]{Co}, the system (\ref{teq2}) is locally trajectory equivalent at the origin to the system
$$
\frac{dz}{dt}=b z^{k+1}(t).
$$
Hence, it follows from \cite[Theorem 1]{S} that $\gamma(t)\to 0$ as $t\to +\infty$. 

Now we shall estimate $\gamma(t)$. Indeed, integration shows that
\begin{equation}\label{eq???1}
\frac{1}{\gamma^k(t)}=c-kbt\big(1+\epsilon(t)\big),~\forall t>t_0,
\end{equation}
where $c$ is a constant depending only on the initial condition and  
\begin{equation*}
\epsilon(t)= \frac{\int_{t_0}^t g_o(\gamma(s))ds}{t-t_0} ~\text{for every}~t>t_0.
\end{equation*}

Choose $\delta>0$ such that either $\arg\big(b(1+z)\big)\in (0,2\pi)$ for all $z\in \Delta_\delta$ (for the case $\mathrm{Im}(b)\ne 0$) or $\arg\big(b(1+z)\big)\in (-\pi/2,3\pi/2)$ for all $z\in \Delta_\delta$ (for the case $\mathrm{Im}(b)=0$). Moreover, without loss of generality we can assume that $|g_0(\gamma(t))|<\delta$ for all $t> t_0$ and hence $|\epsilon(t)|<\delta$ for all $t>t_0$. Therefore, by changing the initial condition $\gamma(t_0)=z_0$ if necessary, we may assume that either
$c-kbt(1+\epsilon(t)),c-kbt\in \mathbb C\setminus [0,+\infty)$ for all $t\in [t_0, +\infty)$ or $c-kbt(1+\epsilon(t)),c-kbt\in \mathbb C\setminus (-\infty i,0]$ for all $t\in  [t_0, +\infty)$. Without loss of generality, we can assume that the first case occurs. 

Notice that $\omega_j(t):=\tau^{-j}\sqrt[-k]{c-kbt},~j=0,\ldots,k-1,$ are solutions of the equation 
$$
\frac{dz}{dt}=bz^{k+1},
$$
where $\tau:=e^{i2\pi/k}$. Furthermore, for each $j\in \{0,1,\ldots,k-1\}$ let $\theta_j(t)~(t\geq t_0)$ be the solution of the equation
$$
\theta_j'(t)=f(\omega_j(t)+\theta_j(t))-b\omega^{k+1}_j(t)
$$
satisfying $\theta_j(t_0)=0$. Then $\gamma_j(t):=\omega_j(t)+\theta_j(t)~(t>t_0),~ j=0,1,\ldots,k-1$, are solutions of 
$$  
 \frac{dz}{dt}=f(z).
$$ 
Moreover, again by changing the initial condition $\gamma(t_0)=z_0$ if necessary we can assume that $|g_0(\gamma_j(t))|<\delta$ for every $j=0,1,\ldots,k-1$ and for every $t>t_0$. In addition, integeration shows that
\begin{equation}\label{eq???2}
\frac{1}{\gamma_j^k(t)}=c-kbt\big(1+\epsilon_j(t)\big),~\forall t>t_0,
\end{equation}
where $\epsilon_j(t)=\frac{\int_{t_0}^t g_o(\gamma_j(s))ds}{t-t_0}$ for every $t>t_0$ and for every $j=0,1,\ldots,k-1$. Hence, we obtain the following. 
\begin{equation*}
\begin{split}
\gamma_j(t)&=\tau^{-j}\sqrt[-k]{c-kbt\big(1+\epsilon_j(t)\big)}\\
        &= \tau^{-j}\sqrt[-k]{|c-kbt(1+\epsilon(t))|} e^{-i \arg \big(c-kbt(1+\epsilon(t))\big)/k}\\
         &= \sqrt[-k]{|c-kbt(1+\epsilon(t))|} e^{-i \arg \big(c-kbt(1+\epsilon(t))\big)/k-i2\pi j/k}, 
\end{split}
\end{equation*}
where $0<\arg\big(c-kbt(1+\epsilon_j(t))\big)/k<2\pi$, for every $j=0,1,\ldots,k-1$. Consequently, $|\gamma_j(t)|\approx \dfrac{1}{|t|^{1/k}}$ for all $t\geq t_0$ and for all $j=0,1,\ldots,k-1$.

Let $u_j(t):= F(\gamma_j(t))$ for $j=0,1,\ldots,k-1$. It follows from (A2) that
\begin{equation}\label{eq77}
u_j'(t)=-P^n(\gamma_j(t))\Big(\text{Re} \big(a \gamma_j^m(t) +o(|\gamma_j(t)|^m)\big)\Big)+O(|\gamma_j(t)|^{k+1})
\end{equation}
for all $t>t_0$ and for all $j=0,1,\ldots,k-1$.

We now consider the following.
\medskip

%\noindent
\underbar{\it Subcase 3.1}: $n\geq 1.$

Since $P$ is flat at the origin, (\ref{eq77}) and the discussion above imply
\begin{equation*}
\begin{split}
|u_0'(t)|&\lesssim P^n(\gamma_0(t)) |\gamma_0(t)|^m+\frac{1}{t^{1+1/k}}\\
&\lesssim P^n(\gamma_0(t))+\frac{1}{t^{1+1/k}}\\
&\lesssim \frac{P^n(\gamma_0(t))}{|\gamma_0(t)|^{2k}} \frac{1}{t^2}+\frac{1}{t^{1+1/k}}\\
&\lesssim  \frac{1}{t^2}+\frac{1}{t^{1+1/k}} \\
&\lesssim  \frac{1}{t^{1+1/k}}
\end{split}
\end{equation*}
for all $t\geq t_0$. This in turn yields
\begin{equation*}
\begin{split}
|u_0(t)|&\lesssim |u_0(t_0)|+ \int_{t_0}^t   \frac{1}{s^{1+1/k}}ds\\
 &\lesssim |u_0(t_0)|
+k \Big(\frac{1}{t_0^{1/k}}-\frac{1}{t^{1/k}}\Big)\\
&\lesssim 1
\end{split}
\end{equation*}
for all $t>t_0$. This is a contradiction, because $\lim_{t\to\infty} u_0(t)=-\infty$.
\medskip

\noindent
\underbar{\it Subcase 3.2}: $n=0$.

We again divide the argument into 4 sub-subcases.
\medskip

\underbar{\it Subcase 3.2.1}: $m/k>1$.

It follows from (\ref{eq77}) that
\begin{equation*}
|u_0'(t)|\lesssim  \frac{1}{t^{m/k}}+\frac{1}{t^{1+1/k}}\\
\end{equation*}
for all $t\geq t_0$. Hence, we get

\begin{equation*}
\begin{split}
|u_0(t)|&\lesssim |u_0(t_0)|+ \int_{t_0}^t   \Big(\frac{1}{s^{m/k}}+\frac{1}{s^{1+1/k}}\Big)ds\\
 &\lesssim |u_0(t_0)|+ \frac{k}{m-k}  \Big(\frac{1}{t_0^{m/k-1}}-\frac{1}{t^{m/k-1}}\Big)
+k \Big(\frac{1}{t_0^{1/k}}-\frac{1}{t^{1/k}}\Big)\\
&\lesssim 1
\end{split}
\end{equation*}
for all $t>t_0$, which contradicts $\lim_{t\to +\infty}u_0(t)=-\infty$.
\medskip

\underbar{\it Subcase 3.2.2}: ${m}/{k}=1$.

Here, (\ref{eq77}) again implies
\begin{equation*}
|u_0'(t)| \lesssim  \frac{1}{t}+\frac{1}{t^{1+1/k}} \lesssim  \frac{1}{t}
\end{equation*}
for all $t\geq t_0$. Consequently,
\begin{equation*}
\begin{split}
|u_0(t)|&\lesssim |u_0(t_0)|+ \int_{t_0}^t   \frac{1}{s}\ ds\\
 &\lesssim |u_0(t_0)|+  (\log t-\log t_0)\\
 &\lesssim \log t \\
 &\lesssim \log \frac{1}{|\gamma_0(t)|}
\end{split}
\end{equation*}
for all $t>t_0$. Therefore there exists a constant $A>0$ such that
$ |u_0(t)|\leq A\log \dfrac{1}{|\gamma_0(t)|}$ for all $t>t_0$.
Hence for all $t>t_0$,
$\log P(\gamma_0 (t))=2u(t)\geq -2A \log \dfrac{1}{|\gamma_0(t)|} $, and thus
$$
P(\gamma_0(t))\geq |\gamma_0(t)|^{2A}, \; \forall t\geq t_0.
$$
This ensures
$$ \lim_{t\to +\infty} \frac{P(\gamma_0(t))}{|\gamma_0(t)|^{2A+1}}=+\infty,$$
which is again impossible since $P$ is flat at $0$.
\medskip

\underbar{\it Subcase 3.2.3}: $m=0$.

Let $h(t):=u_0(t)+\text{Re}(a) t$. Recall that in this case we have (E2) which says $\text{Re }a \neq 0$. Assume momentarily that $\text{Re}(a)<0$. (The case that $\text{Re}(a)>0$ will follow by a similar argument.)

By (\ref{eq77}), there is a constant $B>0$ such that
$$
|h'(t)|\leq \frac{1}{2}|\text{Re}(a)| +B\frac{1}{t^{1+1/k}}.
$$
Therefore,
\begin{equation*}
\begin{split}
|h(t)|&\leq |h(t_0)|+ \frac{1}{2}|\text{Re}(a)| (t-t_0)+ B\int_{t_0}^t   \frac{1}{s^{1+1/k}}ds\\
 &\leq |h(t_0)|+ \frac{1}{2}|\text{Re}(a)| (t-t_0)+k B(\frac{1}{t_0^{1/k}}-\frac{1}{t^{1/k}})\\
\end{split}
\end{equation*}
for all $t>t_0$. Thus
\begin{equation*}
\begin{split}
u_0(t) &\geq -\text{Re}(a) t-|h(t)|\\
    &\geq |\text{Re}(a)| t-|h(t_0)|- \frac{1}{2}|\text{Re}(a)| (t-t_0)-
k B(\frac{1}{t_0^{1/k}}-\frac{1}{t^{1/k}})\\
     &\gtrsim t
\end{split}
\end{equation*}
for all $t>t_0$. It means that $u_0(t)\to +\infty$ as $t\to +\infty$, and it is hence absurd.
\medskip

\underbar{\it Subcase 3.2.4}: $0<\frac{m}{k}<1$.
Assume for a moment that $m$ and $k$ are relatively prime. (In the end, it will become obvious that this assumption can be taken without loss of generality.) Then $\tau^{m}$ is a primitive $k$-th root of unity. Therefore there
exist $j_0,j_1\in \{1,\cdots,k-1\} $ such that $\pi/2<arg(\tau^{mj_0})\leq\pi$ and
$-\pi\leq arg(\tau^{mj_1})<-\pi/2$. Hence, it follows that there exists
$j\in \{0,\cdots,k-1\}$ such that $\cos\big(arg(a/b)+\frac{k-m}{k} arg(-b)-2\pi mj/k\big)>0$.
Denote by
$$
A:=\frac{|a|}{(k-m)|b|}\cos\Big(arg(a/b)+ \frac{k-m}{k}arg(-b)-2\pi m j/k\Big)>0,
$$
a positive constant. Now let
$$
h_j(t):= u_j(t)+\text{Re}(\tau^{-mj} \frac{a}{-b(k-m)}(c-kbt)^{1-m/k}).
$$
Note that $arg\big(c-kbt)\big)\to arg (-b) $ as $t\to +\infty$ and $\delta>0$ can be chosen so small that there exists $t_1>t_0$ big enough such that
\begin{equation*}
\begin{split}
\Big|\gamma_j^m(t)-\tau^{-mj}\Big(\frac{1}{c-kbt}\Big)^{m/k}\Big|&=\Big|\frac{1}{\big(c-kb(t+\epsilon_j(t))\big)^{m/k}}\Big[1-\Big(1-\frac{kbt\epsilon_j(t)}{c-kbt}\Big)^{m/k}\Big]\Big| \\
&\leq \frac{k-m}{8k|a|} A (k|b|)^{1-m/k}\frac{1}{t^{m/k}}
\end{split}
\end{equation*}
for every $t>t_1$. Hence it follows from (\ref{eq77}) that there exist positive constants $B$ and $t_2~(t_2>t_1)$ such that
$$
|h_j'(t)|\leq \frac{k-m}{4k} A (k|b|)^{1-m/k}\frac{1}{t^{m/k}}+\frac{B}{t^{1+1/k}}
$$
and
\begin{multline*}
\cos\Big(\arg(a/b)+\frac{k-m}{k}\arg(c-kbt)-2mj\pi /k\Big) \\
\geq \frac{1}{2}\cos\Big(\arg(a/b)+ \frac{k-m}{k}\arg(-b)-2mj \pi /k\Big)
\end{multline*}
for every $t\geq t_2$. Thus we have
 \begin{equation*}
\begin{split}
|h_j(t)|&\leq |h_j(t_2)|+A(k|b|)^{1-m/k}\frac{k-m}{4k} \int_{t_2}^t s^{-m/k}ds
+B\int_{t_2}^ts^{-1-1/k}ds\\
        & \leq |h_j(t_2)|+\frac{A}{4}(k|b|)^{1-m/k} (t^{1-m/k}-t_2^{1-m/k})
+kB(t_2^{-1/k}-t^{-1/k})
\end{split}
\end{equation*}
for $t>t_2$. Hence
\begin{equation*}
\begin{split}
u_j(t)&\geq -\text{Re} (\frac{a \tau^{-mj}}{-kb(1-m/k)} (c-kbt)^{1-m/k}) -|h_j(t)|\\
& \geq \frac{|a|}{|b|(k-m))} |c-kbt|^{1-m/k}\cos\Big( arg(a/b)\\
&\qquad +\frac{(k-m)arg(c-kbt)- 2mj\pi }{k}\Big)- |h_j(t_2)|\\
&\qquad -\frac{A}{4} (k|b|)^{1-m/k}(t^{1-m/k}-t_2^{1-m/k})-kB(t_2^{-1/k}-t^{-1/k})\\
&\geq \frac{A}{2}|c-kbt|^{1-m/k}- |h_j(t_2)|\\
&\qquad -\frac{A}{4}(k|b|)^{1-m/k} (t^{1-m/k}-t_2^{1-m/k})-kB(t_2^{-1/k}-t^{-1/k})\\
&\gtrsim t^{1-m/k}
\end{split}
\end{equation*}
for $t>t_2$.  This implies that $u_j(t)\to +\infty$ as $t\to +\infty$,
which is absurd since $\log P(z)\to -\infty$ as $z \to 0$.
\medskip

Hence all the cases are covered, and the proof of Lemma \ref{Al3} is finally complete.
\end{proof}

Following the proof of Lemma \ref{Al3}, we have the following lemma.
\begin{lemma}\label{cor3} Let $P:\Delta_{\epsilon_0}\to \mathbb R$ be a $\mathcal{C}^1$-smooth function satisfying that $P(z)>0$ for any $z\in \Delta^*_{\epsilon_0}$ and that $P$ is flat at $0$. If $b$ is a complex number and if $g$ is a $\mathcal{C}^1$-smooth function defined on $\Delta_{\epsilon_0}$ satisfying:
\begin{itemize}
\item[(B1)] $g(z) = O(|z|^{k+1})$, and
\item[(B2)] $\text{Re} \Big[\big(b z^k+g(z)\big)P'(z)\Big]= 0$ for every $z \in \Delta_{\epsilon_0}$
\end{itemize}
for some nonnegative integer $k$, except the case $k= 1$ and $\text{Re} (b)=0$, then $b=0$.
\end{lemma}
\section{The vetor space of tangential holomorphic vetor fields}
This section is devoted to the proof of Theorem \ref{Theorem3}. First of all, we need the following theorem.
\begin{theorem} \label{Th1}
If a holomorphic vector field germ $(H,0)$ vanishing at the origin which contains no nonzero term $i\beta z_2\frac{\partial}{\partial z_2}~(\beta\in \mathbb R^*)$ and is tangent to
a real $\mathcal{C}^1$-smooth hypersurface germ $(M,0)$ defined by the equation
$\rho(z) := \rho(z_1,z_2)=\mathrm{Re}~z_1+P(z_2)+ \mathrm{Im}~z_1 Q(z_2,\mathrm{Im}~z_1)=0$
satisfying the conditions: 
\begin{itemize}
\item[(1)] $P,Q$ are $\mathcal{C}^1$-smooth with $P(0)=Q(0,0)=0$,
\item[(2)] $P(z_2)>0$ for any $z_2 \not= 0$, and 
\item[(3)] $P(z_2),P'(z_2)$ are flat at $z_2=0$,
\end{itemize}
then $H=0$.
\end{theorem}
\begin{proof}
The CR hypersurface germ $(M,0)$ at the origin in $\CC^2$ under consideration is defined by the equation $\rho(z_1, z_2) = 0$, where
$$
\rho (z_1, z_2) = \mathrm{Re}~z_1 + P(z_2) + (\mathrm{Im}~z_1)\ Q(z_2, \mathrm{Im}~z_1) = 0,
$$
where $P, Q$ are $\mathcal{C}^1$-smooth functions satisfying the three conditions specified in the hypothesis of our lemma.  Recall that $P(z_2),P'(z_2)$ are flat at $z_2=0$ in particular.

Then we consider a holomorphic vector field $H=h_1(z_1,z_2)\frac{\partial}{\partial z_1}+h_2(z_1,z_2)\frac{\partial}{\partial z_2}$ defined on a neighborhood of the origin satisfying that $H(0)=0$ and that $H$ contains no nonzero term $i\beta z_2\frac{\partial}{\partial z_2}~(\beta\in \mathbb R^*)$. We only consider $H$ that is tangent to $M$, which means that they satisfy the identity
\begin{equation}\label{eq221}
(\mathrm{Re}~ H) \rho(z)=0,\; \forall z \in M.
\end{equation}

Expand $h_1$ and $h_2$ into the Taylor series at the origin so that
$$
h_1(z_1,z_2)=\sum\limits_{j,k=0}^\infty a_{jk} z_1^j z_2^k
\text{ and }
h_2(z_1,z_2)=\sum\limits_{j,k=0}^\infty b_{jk} z_1^jz_2^k,
$$
where $a_{jk}, b_{jk}\in \mathbb C$. We note that $a_{00}=b_{00}=0$ since $h_1(0,0)=h_2(0,0)=0$.

By a simple computation, we have 
\begin{equation*}
\begin{split}
 \rho_{z_1}(z_1,z_2)&= \frac{1}{2}+\frac{Q(z_2, \mathrm{Im}~z_1)}{2i}+(\mathrm{Im}~z_1)Q_{z_1}(z_2, \mathrm{Im}~z_1),\\
\rho_{z_2}(z_1,z_2)&= P'(z_2)+(\text{Im}~z_1) Q_{z_2}(z_2, \text{Im}~z_1),
\end{split}
\end{equation*}
and the equation (\ref{eq221}) can thus be re-written as
\begin{equation}\label{eq222}
\begin{split}
&\mathrm{Re} \Big[\Big( \frac{1}{2}+\frac{Q(z_2, \mathrm{Im}~z_1)}{2i}+(\mathrm{Im}~z_1)Q_{z_1}(z_2, \mathrm{Im}~z_1)\Big)h_1(z_1,z_2)\\
&\quad +\Big(P'(z_2)+(\text{Im}~z_1)Q_{z_2}(z_2, \text{Im}~z_1)\Big) h_2(z_1,z_2)\Big ]=0
\end{split}
\end{equation}
for all $(z_1,z_2)\in M$.

Since $\Big(it-P(z_2)-tQ(z_2,t), z_2\Big)\in M$ for any $t \in \mathbb R$ with $t$ small enough, the above equation again admits a new form
\begin{equation}\label{eq223}
\begin{split}
&\mathrm{Re}\Big[ \Big(\frac{1}{2}+\frac{Q(z_2,t)}{2i}+tQ_{z_1}(z_2,t)\Big)\sum_{j,k=0}^\infty a_{jk}\big(it-P(z_2)-tQ(z_2,t)\big)^j z_2^k\\
&\quad +\Big(P'(z_2)+t {Q}_{z_2}(z_2,t)\Big) \sum_{m,n=0}^\infty b_{mn} \big(it-P(z_2)-tQ(z_2,t)\big)^m z_2^n\Big]=0
\end{split}
\end{equation}
for all $z_2\in \mathbb C$ and for all $t\in\mathbb R$ with $|z_2|<\epsilon_0$ and $|t|<\delta_0$, where $\epsilon_0>0$ and $\delta_0>0$ are small enough.

The goal is to show that $H\equiv 0$. Indeed, striving for a contradiction, suppose that $H\not\equiv 0$. We notice that if $h_2\equiv 0$ then (\ref{eq222}) shows that $h_1\equiv 0$. So, we must have  $h_2\not\equiv 0$.

We now divide the argument into two cases as follows.
\smallskip

\noindent
{\bf Case 1.} {\boldmath $h_1\not \equiv 0$.} In this case let us denote by $j_0$ the smallest integer such that $a_{j_0 k}\ne 0$ for some integer $k$. Then let $k_0$ be the smallest integer such that $a_{j_0 k_0}\ne 0$. Similarly, let $m_0$ be the smallest integer such that $b_{m_0 n}\ne 0$ for some integer $n$. Then denote by $n_0$ the smallest integer  such that $b_{m_0 n_0}\ne 0$. We can see that $j_0\geq 1$ if $k_0=0$ and $m_0\geq 1$ if $n_0=0$. 

Since $P(z_2)=o(|z_2|^j)$ for any $j\in \mathbb N$, inserting $t=\alpha P(z_2)$ into (\ref{eq223}), where $\alpha\in \mathbb R$ will be chosen later, one has
\begin{equation}\label{eq224}
\begin{split}
&\mathrm{Re} \Big[\frac{1}{2} a_{j_0k_0}(i\alpha -1)^{j_0}(P(z_2))^{j_0}\big(z_2^{k_0}+o(|z_2|^{k_0})\big)+ b_{m_0n_0}(i\alpha -1)^{m_0}(z_2^{n_0}+o(|z_2|^{n_0}) \\
&\quad \times (P(z_2))^{m_0}\Big(P'(z_2)+\alpha P(z_2)Q_{z_2}(z_2, \alpha P(z_2))\Big)   \Big ]=0
\end{split}
 \end{equation}
for all $z_2\in \Delta_{\epsilon_0}$.  We note that in the case $k_0=0$ and $\mathrm{Re}(a_{j_0 0})=0$, $\alpha$ can be chosen in such a way that $\mathrm{Re}\big( (i\alpha-1)^{j_0}a_{j_0 0}\big)\ne 0$. Then (\ref{eq224}) yields that $j_0>m_0$ by virtue of the fact that $P'(z_2), P(z_2)$ are flat at $z_2=0$. Hence, we conclude from Lemma \ref{Al3} that $m_0=0,n_0=1$, and $b_{0,1}=i\beta z_2$ for some $\beta\in \mathbb R^*$. This is a contradiction with the assumption $H$ contains no nonzero term $i\beta z_2\frac{\partial}{\partial z_2}$.

\medskip

\noindent
{\bf Case 2.} {\boldmath $h_1\equiv 0$.} Let $m_0,n_0$ be as in the Case 1. Since $P(z_2)=o(|z_2|^{n_0})$, letting $t=0$ in (\ref{eq223}) one obtains that
\begin{equation}\label{eq225}
\begin{split}
\mathrm{Re} \Big[b_{m_0n_0}\big(z_2^{n_0}+o(|z_2|^{n_0}\big) P'(z_2) \Big ]=0
\end{split}
 \end{equation}
for all $z_2\in \Delta_{\epsilon_0}$. Therefore, Lemma \ref{cor3} yields that $m_0=0,n_0=1$, and $b_{0,1}=i\beta z_2$ for some $\beta\in \mathbb R^*$, which is again impossible.

Altogether, the proof of our theorem is complete. 
\end{proof}

Now we are ready to prove Theorem \ref{Theorem3}. 
\begin{proof}[Proof of Theorem \ref{Theorem3}]
Let $H_1, H_2\in \mathrm{hol}_0(M,p) $ be arbitrary. Then by Theorem \ref{Th1} we have that $H_j$ contains term $i\beta_jz_2\frac{\partial}{\partial z_2}$ ($j=1,2$) for some $\beta_1,\beta_2\in \mathbb R$. Therefore, $\beta_2H_1-\beta_1 H_2$ does not contain a term $i\beta z_2\frac{\partial}{\partial z_2}$. Hence, Theorem \ref{Th1} again yields that $\beta_2H_1-\beta_1 H_2=0$, which proves the theorem.  
\end{proof}

\section{Appendix}

We recall the following theorem that gives examples of holomorphic vector fields and real hypersurfaces which are tangent.
\begin{theorem}[see Theorem 3 in \cite{Ninh}]\label{T2} Let $\alpha \in \mathbb R$ and let $a(z)= \sum_{n=1}^\infty a_n z^n $ be a non-zero holomorphic function defined on a neighborhood of $0\in \mathbb C$, where $a_n \in \mathbb C$ for all $n\geq 1$. Then there exist positive numbers $\epsilon_0,\delta_0>0$ such that the holomorphic vector field 
$$
H^{a,\alpha}(z_1,z_2)=L^\alpha(z_1) a(z_2)\frac{\partial }{\partial z_1}+i z_2\frac{\partial }{\partial z_2},
$$
where 
\[
L^\alpha (z_1)=
\begin{cases}
\frac{1}{\alpha}\big(\exp(\alpha z_1)-1\big)&\text{if}~ \alpha \ne 0\\
z_1 &\text{if}~ \alpha =0,
\end{cases}
\]
is tangent to the $\mathcal{C}^1$-smooth hypersurface $M$ given by
$$
M=\big\{ (z_1,z_2)\in \Delta_{\delta_0}\times \Delta_{\epsilon_0}:\rho(z_1,z_2):= \mathrm{Re}~z_1+P(z_2)+f(z_2,\mathrm{Im}~z_1) =0\big\},
$$
where $f$ and $P$ are respectively defined on $\Delta_{\epsilon_0}\times (-\delta_0,\delta_0)$ and $\Delta_{\epsilon_0}$ by
\[
 f(z_2,t)=\begin{cases}
 -\frac{1}{\alpha}\log \Big|\frac{\cos \big(R(z_2)+\alpha t\big)}{\cos (R(z_2))} \Big| &~\text{if}~ \alpha\ne 0\\
 \tan(R(z_2))t  &~\text{if}~ \alpha =0,
\end{cases} 
 \]
 where $R(z_2)=q(|z_2|)- \mathrm{Re}\big(\sum_{n=1}^\infty\frac{a_n}{n} z_2^n\big)$ for all $z_2\in \Delta_{\epsilon_0}$,
and
\begin{equation*}
\begin{split}
  P(z_2)=
 \begin{cases}
\frac{1}{\alpha} \log \Big[ 1+\alpha P_1(z_2)\Big]~&\text{if}~ \alpha \ne 0\\
 P_1(z_2) ~&\text{if}~ \alpha=0,
\end{cases}
\end{split}
\end{equation*}
where 
\begin{equation*}
\begin{split}
P_1(z_2)=\exp\Big[p(|z_2|)+\mathrm{Re}\Big(\sum_{n=1}^\infty  \frac{a_n}{in}z_2^n\Big ) -\log \big|\cos\big(R(z_2)\big)\big|   \Big]
\end{split}
\end{equation*}
for all $z_2\in \Delta_{\epsilon_0}^*$ and $P_1(0)=0$, and $q,p$ are reasonable functions defined on $[0,\epsilon_0)$ and $(0,\epsilon_0)$ respectively with $q(0)=0$ so that $P,R$ are $\mathcal{C}^1$-smooth in $\Delta_{\epsilon_0}$. 
\end{theorem}
\begin{proof}
First of all, it is easy to show that there is $\epsilon_0>0$ such that we can choose a function $q$ so that the function $R$ defined as in the theorem is $\mathcal{C}^1$-smooth and $|R(z_2)|\leq 1$ on $\Delta_{\epsilon_0}$. Choose $\delta_0=\frac{1}{2|\alpha|}$ if $\alpha\ne 0$ and  $\delta_0=+\infty$ if otherwise. Then the function $ f(z_2,t)$ given in the theorem is $\mathcal{C}^1$-smooth on $\Delta_{\epsilon_0}\times (-\delta_0,\delta_0)$. Moreover, $ f(z_2,t)$ is real analytic in $t$ and $\frac{\partial^mf}{\partial t^m}$ is $\mathcal{C}^1$-smooth on $\Delta_{\epsilon_0}\times (-\delta_0,\delta_0)$ for each $m\in \mathbb N$.

Next, let $P_1, P, R$ be functions defined as in the theorem and let $Q_0(z_2):=\tan (R(z_2))$ for all $z_2 \in \Delta_{\epsilon_0}$. By a direct computation, we have the following equations.
\begin{itemize}
\item[(i)] $\mathrm{Re}\Big[i  z_2 {Q_0}_{z_2}(z_2)+\frac{1}{2}\Big(1+Q_0^2(z_2)\Big)i a(z_2)\Big]\equiv 0$;
\item[(ii)] $\mathrm{Re}\Big[i z_2 {P_1}_{z_2}(z_2)-\Big(\frac{1}{2}+\frac{Q_0(z_2)}{2i}\Big)a(z_2)P_1(z_2)\Big]\equiv 0$;
\item[(iii)] $\mathrm{Re}\Big[i z_2 {P}_{z_2}(z_2)+\frac{\exp\big(-\alpha P(z_2)\big)-1}{\alpha}\Big(\frac{1}{2}+\frac{Q_0(z_2)}{2i}\Big) a(z_2)\Big]\equiv 0 $ for $\alpha\ne 0$;
\item[(iv)] $\Big(i+f_t(z_2,t\Big)\exp\Big(\alpha\big(i t-f(z_2,t)\big)\Big)\equiv i+Q_0(z_2)$;
\item[(v)] $\mathrm{Re}\Big[2i\alpha z_2 f_{z_2}(z_2,t)+\Big(f_t(z_2,t)-Q_0(z_2)\Big) ia(z_2)\Big]\equiv 0$
\end{itemize}
on $\Delta_{\epsilon_0}$ for any $t\in (-\delta_0,\delta_0)$.

We now prove that the holomorphic vector field $H^{a,\alpha}$ is tangent to the hypersurface $M$. Indeed, by a calculation we get
\begin{equation*}
\begin{split}
 \rho_{z_1}(z_1,z_2)&=\frac{1}{2}+\frac{f_t(z_2,\mathrm{Im}~z_1 )}{2i},\\
 \rho_{z_2}(z_1,z_2)&=P_{z_2}(z_2)+f_{z_2}(z_2,\mathrm{Im}~z_1). 
 \end{split}
 \end{equation*}
 
We divide the proof into two cases.
  
\noindent 
 {\bf a) $\alpha =0$.} In this case, $f(z_2,t)=Q_0(z_2)t$ for all $(z_2,t)\in \Delta_{\epsilon_0}\times (-\delta_0,\delta_0)$. Therefore, by $\mathrm{(i)}$ and $\mathrm{(ii)}$ one obtains that
 \begin{equation*}
 \begin{split}
\mathrm{Re}~H^{a,\alpha}(\rho(z_1,z_2))&=\mathrm{Re}\Big[\Big(\frac{1}{2}+\frac{Q_0(z_2)}{2i}\Big) z_1a(z_2)
+\Big({P_1}_{z_2}(z_2)+(\mathrm{Im}~z_1){Q_0}_{z_2}(z_2)\Big) i z_2\Big]\\
&= \mathrm{Re}\Big[\Big(\frac{1}{2}+\frac{Q_0(z_2)}{2i}\Big) \Big(i (\mathrm{Im}~z_1)-P_1(z_2)-
(\mathrm{Im}~z_1) Q_0(z_2)\Big)a(z_2)\\
&\quad\quad +\Big({P_1}_{z_2}(z_2)+(\mathrm{Im}~z_1){Q_0}_{z_2}(z_2)\Big) i z_2\Big]\\
&=\mathrm{Re}\Big[i z_2 {P_1}_{z_2}(z_2) -\Big(\frac{1}{2}+\frac{Q_0(z_2)}{2i}\Big)a(z_2) P_1(z_2)\Big]\\
&\quad \quad+(\mathrm{Im}~z_1)\mathrm{Re}\Big[i z_2 {Q_0}_{z_2}(z_2) +\frac{1}{2}\Big(1+Q_0(z_2)^2\Big)ia(z_2)\Big]=0
\end{split}
 \end{equation*}
 for every $(z_1,z_2)\in M$, which proves the theorem for $\alpha =0$.
   
\noindent
{\bf b) $\alpha\ne 0$.} It follows from $\mathrm{(iii)}$, $\mathrm{(iv)}$, and $\mathrm{(v)}$ that
\begin{equation*}
\begin{split}
&\mathrm{Re}~H^{a,\alpha}(\rho(z_1,z_2))\\
&=\mathrm{Re}\Big[\Big(\frac{1}{2}+\frac{f_t(z_2,\mathrm{Im}~z_1)}{2i}\Big)L(z_1)a(z_2)
+\Big(P_{z_2}(z_2)+f_{z_2}(z_2,\mathrm{Im}~z_1)\Big) i z_2\Big]\\
&=\mathrm{Re}\Big[\Big(\frac{1}{2}+\frac{f_t(z_2,\mathrm{Im}~z_1)}{2i}\Big)\frac{1}{\alpha}\Big(\exp\Big(\alpha\big(i\mathrm{Im}~z_1-P(z_2)-
f(z_2,\mathrm{Im}~z_1)\big)\Big)-1 \Big)a(z_2) \\
&\quad +\Big(P_{z_2}(z_2)+f_{z_2}(z_2,\mathrm{Im}~z_1)\Big) i z_2\Big]\\
&=\mathrm{Re}\Big[\frac{1}{\alpha}\frac{i+f_t(z_2,\mathrm{Im}~z_1)}{2i}\exp\Big(\alpha\big(i\mathrm{Im}~z_1-
f(z_2,\mathrm{Im}~z_1)\big)\Big)\exp(-\alpha P(z_2) )a(z_2) \\
&\quad - \frac{1}{\alpha}\Big(\frac{1}{2}+\frac{f_t(z_2,\mathrm{Im}~z_1)}{2i}\Big)a(z_2) +\Big(P_{z_2}(z_2)+f_{z_2}(z_2,\mathrm{Im}~z_1)\Big) i z_2\Big]\\
&=\mathrm{Re}\Big[\frac{1}{\alpha}\frac{i+Q_0(z_2)}{2i}\exp(-\alpha P(z_2) )a(z_2)- \frac{1}{\alpha}\Big(\frac{1}{2}+\frac{f_t(z_2,\mathrm{Im}~z_1)}{2i}\Big)a(z_2) \\
&\quad  +\Big(P_{z_2}(z_2)+f_{z_2}(z_2,\mathrm{Im}~z_1)\Big) i z_2\Big]\\
&=\mathrm{Re}\Big[i z_2P_{z_2}(z_2)+\Big(\frac{1}{2}+\frac{Q_0(z_2)}{2i}\Big)\frac {\exp(-\alpha P(z_2))-1}{\alpha}a(z_2) \Big]\\
&\quad+ \mathrm{Re}\Big[ i z_2 f_{z_2}(z_2,\mathrm{Im}~z_1)+ \frac{1}{2 \alpha}\Big(f_t(z_2,\mathrm{Im}~z_1)-Q_0(z_2)\Big)ia(z_2)\Big]=0
\end{split}
\end{equation*}
 for every $(z_1,z_2)\in M$, which ends the proof. 
\end{proof}


\begin{thebibliography}{99}

\bibitem{Baouendi1} M. S. Baouendi, P. Ebenfelt and L. P. Rothschild, Real submanifolds in complex space and their mappings, Princeton Math. Series, 47. Princeton Univ. Press, Princeton, NJ, 1999.

\bibitem{By1} J. Byun, J.-C. Joo and M. Song, The characterization
of holomorphic vector fields vanishing at an infinite type point,
J. Math. Anal. Appl. 387 (2012), 667--675.
\bibitem{Chern} S. S. Chern and  J. K. Moser, Real hypersurfaces in complex manifolds, Acta Math. 133 (1974), 219--271. 

\bibitem{Co} C. Coleman, Equivalence of planar dynamical and differential systems, J. Differential Equations 1 (1965),  222--233.

\bibitem{D} J. P. D'Angelo, Real hypersurfaces, orders of contact, and applications, Ann. Math. 115 (1982), 615--637.

\bibitem{GGJ} A. Garijo, A. Gasull and X. Jarque, Local and global phase portrait of equation $\dot z=f(z)$, Discrete Contin. Dyn. Syst. 17 (2) (2007), 309--329.

\bibitem{Kolar1} M. Kol\'{a}\v{r} and F. Meylan, Infinitesimal CR automorphisms of hypersurfaces of finite type in $\mathbb C^2$,  Arch. Math. (Brno) 47 (5) (2011), 367--375.

\bibitem{Kim-Ninh} K.-T. Kim and V. T. Ninh, On the tangential holomorphic vector fields vanishing at an infinite type point,  arXiv:1206.4132, to appear in Trans. Amer. Math. Soc..

\bibitem{Ninh} V. T. Ninh, On the existence of tangential holomorphic vector fields vanishing at an infinite type point, arXiv:1303.6156.

\bibitem{Stanton1} N. Stanton, Infinitesimal CR automorphisms of real hypersurfaces, Amer. J. Math. 118 (1) (1996), 209--233.
\bibitem{Stanton2} N. Stanton, Infinitesimal CR automorphisms of rigid hypersurfaces, Amer. J. Math. 117 (1)  (1995), 141--167.
\bibitem{S} R. Sverdlove, Vector fields defined by complex functions, J. Differential Equations 34(1979), no. 3, 427--439.

\end{thebibliography}
\end{document}